\documentclass[11pt]{amsart}
\usepackage{amsopn}
\usepackage{amssymb, amscd}
\usepackage{multirow}
\usepackage{enumerate}
\usepackage{tikz}
\usepackage{hyperref}
\usepackage{mathrsfs}

\newcommand{\nc}{\newcommand}

\nc{\fg}{\mathfrak{f} } \nc{\vg}{\mathfrak{v} } \nc{\wg}{\mathfrak{w} } \nc{\zg}{\mathfrak{z} } \nc{\ngo}{\mathfrak{n} } \nc{\kg}{\mathfrak{k} }
\nc{\mg}{\mathfrak{m} } \nc{\bg}{\mathfrak{b} } \nc{\ggo}{\mathfrak{g} } \nc{\ggob}{\overline{\mathfrak{g}} } \nc{\sog}{\mathfrak{so} }
\nc{\sug}{\mathfrak{su} } \nc{\spg}{\mathfrak{sp} } \nc{\slg}{\mathfrak{sl} } \nc{\glg}{\mathfrak{gl} } \nc{\cg}{\mathfrak{c} } \nc{\rg}{\mathfrak{r} }
\nc{\hg}{\mathfrak{h} } \nc{\tg}{\mathfrak{t} } \nc{\ug}{\mathfrak{u} } \nc{\dg}{\mathfrak{d} } \nc{\ag}{\mathfrak{a} } \nc{\pg}{\mathfrak{p} }
\nc{\sg}{\mathfrak{s} } \nc{\affg}{\mathfrak{aff} } \nc{\qg}{\mathfrak{q} }

\nc{\pca}{\mathcal{P}} \nc{\nca}{\mathcal{N}} \nc{\lca}{\mathcal{L}} \nc{\oca}{\mathcal{O}} \nc{\mca}{\mathcal{M}} \nc{\tca}{\mathcal{T}}
\nc{\aca}{\mathcal{A}} \nc{\cca}{\mathcal{C}} \nc{\gca}{\mathcal{G}} \nc{\sca}{\mathcal{S}} \nc{\hca}{\mathcal{H}} \nc{\bca}{\mathcal{B}}
\nc{\dca}{\mathcal{D}} \nc{\ica}{\mathcal{I}} \nc{\val}{\operatorname{val}}

\nc{\vp}{\varphi} \nc{\ddt}{\frac{d}{dt}} \nc{\dds}{\frac{d}{ds}} \nc{\dpar}{\frac{\partial}{\partial t}} \nc{\im}{\mathrm{i}}

\nc{\SO}{\mathrm{SO}} \nc{\Spe}{\mathrm{Sp}} \nc{\Sl}{\mathrm{SL}} \nc{\SU}{\mathrm{SU}} \nc{\Or}{\mathrm{O}} \nc{\U}{\mathrm{U}} \nc{\Gl}{\mathrm{GL}}
\nc{\Se}{\mathrm{S}} \nc{\Cl}{\mathrm{Cl}} \nc{\Spein}{\mathrm{Spin}} \nc{\Pin}{\mathrm{Pin}} \nc{\G}{\mathrm{GL}_n(\RR)} \nc{\g}{\mathfrak{gl}_n(\RR)}

\nc{\RR}{{\Bbb R}} \nc{\HH}{{\Bbb H}} \nc{\CC}{{\Bbb C}} \nc{\ZZ}{{\Bbb Z}} \nc{\FF}{{\Bbb F}} \nc{\NN}{{\Bbb N}} \nc{\QQ}{{\Bbb Q}} \nc{\PP}{{\Bbb P}}
\nc{\OO}{{\Bbb O}}

\nc{\vs}{\vspace{.2cm}} \nc{\vsp}{\vspace{1cm}} \nc{\ip}{\langle\cdot,\cdot\rangle} \nc{\ipp}{(\cdot,\cdot)} \nc{\la}{\langle} \nc{\ra}{\rangle}
\nc{\unm}{\frac{1}{2}} \nc{\unc}{\frac{1}{4}} \nc{\und}{\frac{1}{16}} \nc{\no}{\vs\noindent} \nc{\lam}{\Lambda^2(\RR^n)^*\otimes\RR^n} \nc{\tangz}{{\rm
T}^{\rm Zar}} \nc{\nor}{{\sf n}}  \nc{\mum}{/\!\!/} \nc{\kir}{/\!\!/\!\!/} \nc{\Ri}{\tfrac{4\Ric_{\mu}}{||\mu||^2}} \nc{\ds}{\displaystyle}
\nc{\ben}{\begin{enumerate}} \nc{\een}{\end{enumerate}} \nc{\f}{\frac} \nc{\lb}{[\cdot,\cdot]} \nc{\isn}{\tfrac{1}{||v||^2}}
\nc{\gkp}{(\ggo=\kg\oplus\pg,\ip)} \nc{\ukh}{(\ug=\kg\oplus\hg,\ip)} \nc{\tgkp}{(\tilde{\ggo}=\kg\oplus\pg,\ip)} \nc{\wt}{\widetilde} \nc{\mm}{M}
\nc{\iop}{\mathtt{i}} \nc{\jop}{\mathtt{j}}

\nc{\Hess}{\operatorname{Hess}} \nc{\ad}{\operatorname{ad}} \nc{\Ad}{\operatorname{Ad}} \nc{\rank}{\operatorname{rank}} \nc{\Irr}{\operatorname{Irr}}
\nc{\End}{\operatorname{End}} \nc{\Aut}{\operatorname{Aut}} \nc{\Inn}{\operatorname{Inn}} \nc{\Der}{\operatorname{Der}} \nc{\Ker}{\operatorname{Ker}}
\nc{\Iso}{\operatorname{Iso}} \nc{\Diff}{\operatorname{Diff}} \nc{\Lie}{\operatorname{Lie}} \nc{\tr}{\operatorname{tr}} \nc{\dif}{\operatorname{d}}
\nc{\sen}{\operatorname{sen}} \nc{\modu}{\operatorname{mod}} \nc{\CRic}{\operatorname{PP}} \nc{\Cric}{\operatorname{P}} \nc{\Ricci}{\operatorname{Ric}}
\nc{\sym}{\operatorname{sym}} \nc{\herm}{\operatorname{herm}} \nc{\symac}{\operatorname{sym^{ac}}} \nc{\symc}{\operatorname{sym^{c}}}
\nc{\scalar}{\operatorname{scal}} \nc{\grad}{\operatorname{grad}} \nc{\ricci}{\operatorname{Rc}} \nc{\Nor}{\operatorname{Norm}}
\nc{\ricc}{\operatorname{Rc^{c}}} \nc{\Ricc}{\operatorname{Ric^{c}}} \nc{\ricac}{\operatorname{Rc^{ac}}} \nc{\Ricac}{\operatorname{Ric^{ac}}}
\nc{\Riem}{\operatorname{R}} \nc{\riccig}{\operatorname{ric^{\gamma}}} \nc{\Rin}{\operatorname{M}} \nc{\Le}{\operatorname{L}} \nc{\tang}{\operatorname{T}}
\nc{\level}{\operatorname{level}} \nc{\rad}{\operatorname{r}} \nc{\abel}{\operatorname{ab}} \nc{\CH}{\operatorname{CH}} \nc{\mcc}{\operatorname{mcc}}
\nc{\Adj}{\operatorname{Adj}} \nc{\Order}{\operatorname{O}}  \nc{\inj}{\operatorname{inj}} \nc{\proy}{\operatorname{proy}} \nc{\vol}{\operatorname{vol}}
\nc{\Diag}{\operatorname{Dg}} \nc{\Spec}{\operatorname{Spec}} \nc{\Ima}{\operatorname{Im}} \nc{\Rea}{\operatorname{Re}} \nc{\spann}{\operatorname{sp}}
\nc{\id}{\operatorname{id}}

\theoremstyle{plain}
\newtheorem{theorem}{Theorem}[section]
\newtheorem{proposition}[theorem]{Proposition}
\newtheorem{corollary}[theorem]{Corollary}
\newtheorem{lemma}[theorem]{Lemma}

\theoremstyle{definition}
\newtheorem{definition}[theorem]{Definition}

\theoremstyle{remark}
\newtheorem{remark}[theorem]{Remark}

\newtheorem{example}[theorem]{Example}

\title{New examples of shrinking Laplacian solitons}

\author{Marina Nicolini}

\address{Universidad Nacional de C\'ordoba, FaMAF and CIEM, 5000 C\'ordoba, Argentina}
\email{mnicolini@famaf.unc.edu.ar}

\thanks{This research was partially supported by Universidad Nacional de C\'ordoba.}

\begin{document}

\maketitle

\begin{abstract} We give a one-parameter family of examples of shrinking Laplacian solitons, which are the second known solutions to the closed $G_2$-Laplacian flow with a
finite-time singularity. The torsion forms and the Laplacian and Ricci operators of a large family of $G_2$-structures on different Lie groups are also
studied. We apply these formulas to prove that, under a suitable extra condition, there is no closed eigenform for the Laplacian on such family.
\end{abstract}

\section{Introduction}

On a differentiable $7$-manifold $M$, a $G_2$-structure is a differentiable $3$-form $\vp$ on $M$ such that at each $p\in M$ one can write:
\begin{equation}\label{phi-intro}
\vp_p=e^{127}+e^{347}+e^{567}+e^{135}-e^{146}-e^{236}-e^{245},
\end{equation}
for some basis $\{e_1,\dots,e_7\}$ of $T_pM$. It is known that such a $\vp$ induces a Riemannian metric and an orientation on $M$, and therefore the
corresponding Hodge star and  Hodge Laplacian operators on forms. In \cite{Bry}, Bryant introduced the Laplacian flow for closed $G_2$-structures given by
$$
\left\{\begin{array}{l}\dpar\vp(t) = \Delta\vp(t),\\
 \vp(0)=\vp.
\end{array}\right.
$$
We refer to the recent surveys \cite{Lty,Wei} for accounts of several important results on this flow.  The long-time behavior of solutions is the main
problem, a given solution is expected to converge to a torsion-free $G_2$-structure under appropriate conditions.  However, long-time existence of
solutions is still an open problem in the case when $M$ is compact (see \cite[Section 4.4]{Lty}). On the other hand, in the non-compact case, the only
solutions with a finite-time singularity  known so far are the shrinking Laplacian solitons found in \cite{LS-ERP} on solvable Lie groups.

It is known that a solution $\vp(t)$ flows in a self-similar way, i.e.,
$$
\vp(t)=c(t)f(t)^*\vp, \quad\mbox{for some } c(t)\in\RR^* \mbox{ and } f(t)\in\Diff(M),
$$
if and only if
$$
\Delta\vp=c\vp+\lca_{X}\vp, \qquad \mbox{for some}\quad c\in\RR, \quad X\in\mathfrak{X}(M)\; \mbox{(complete)},
$$
where $\lca_X$ denotes the Lie derivative with respect to the field $X$, in which case $c(t)=\left(\frac{2}{3}ct+1\right)^{3/2}$. In this case, we call
$\vp$ a {\it Laplacian soliton} and we say that it is {\it expanding}, {\it steady} or {\it shrinking}, if $c>0$, $c=0$ or $c<0$, respectively. Note that
in the shrinking case the solution develops a finite-time singularity at $T=-\tfrac{3}{2\,c}>0$.

As mentioned above, previous to this work, there was in the literature only a one-parameter family of shrinking Laplacian solitons, given by Lauret in
\cite[Example 4.10]{LS-ERP} as left-invariant $G_2$-structures on certain solvable Lie groups. In Section \ref{sec-sol}, we provide a new pairwise
non-homothetic family of shrinking Laplacian solitons on solvable Lie groups, which is not equivalent to the family given by Lauret. In this way, we
provide new examples of Laplacian flow solutions that have a finite-time singularity.

Lie groups are a practical tool for the study of $G_2$-structures, since it is sufficient to study the problems at the Lie algebra level. In Section
\ref{sec-ABC}, we fix a $G_2$-structure $\vp\in\Lambda^3\ggo^*$ as in \eqref{phi-intro}, and vary the Lie bracket on $\ggo$ which depends on matrices
$A_1\in\glg_2(\RR)$ and $A,B,C\in\glg_4(\RR)$ in the following way,
$$
A_1=\ad{e_7}|_{\spann\{e_1,e_2\}}=\left[\begin{array}{cc}x&z\\
y&w\end{array}\right],\quad A=\ad{e_7}|_{\ggo_1},\quad B=\ad{e_1}|_{\ggo_1},\quad C=\ad{e_2}|_{\ggo_1},
$$
where $\spann\{e_1,e_2\}$ is abelian, $\spann\{e_7,e_1,e_2\}$ is a subalgebra, $\ggo_1:=\spann\{e_3,e_4,e_5,e_6\}$ is an abelian ideal and
$\spann\{e_1,\dots,e_6\}$ is unimodular, that is, $\tr B=\tr C=0$. We call $G_{A_1,A,B,C}$ the corresponding simply connected Lie group. We compute the
formulas for some operators, such as the Laplacian or the Ricci operator, together with the torsion forms, in terms of the coefficients of the matrices.
The formulas are given in general, beyond the closed and coclosed case, and have already been used in \cite{KL} to study the Laplacian coflow and its
solitons in the case $A_1=0$, $\tr{A}=0$ and $A$,$B$,$C$ symmetric.

A $G_2$-structure $\vp$ that satisfies $\Delta\vp=\lambda\vp$, for some real number $\lambda$, is said to be an \emph{eigenform}. In the compact case,
Lotay and Wei showed in \cite[Proposition 9.2]{LW} that every closed eigenform must be torsion-free. However, it is still an open and intriguing question
if such structures exist in the non-compact case. In Section \ref{sec-EF}, we apply the above mentioned formulas to study this problem on the family of
$G_2$-structures $\{(G_{A_1,A,B,C},\vp)\}$. We prove that if in addition one assumes that the torsion form is given by
$\tau_2=a\,e^{12}+b\,e^{34}+c\,e^{56}$, for some $a+b+c=0$, then $\vp$ must be torsion-free ($\tau_2=0$).

A closed $G_2$-structure that satisfies the following condition:
\begin{equation}\label{ERP-intro}
\Delta\vp=d\tau=\tfrac{1}{6}|\tau|^2+\tfrac{1}{6}\ast(\tau\wedge\tau),
\end{equation}
is called an Extremally Ricci pinched (ERP) $G_2$-structure. In \cite{ERP2}, a complete classification of left-invariant ERP $G_2$-structures on Lie groups
is obtained. Moreover, it is proved that any left-invariant ERP $G_2$-structure on a Lie group is a steady Laplacian soliton and its underlying metric is
an expanding Ricci soliton (i.e.\, a self-similar solution to the Ricci flow $\dpar g(t)=-2\Ricci(g(t))$). The converse is not true, in \cite{FinRff3} the
authors gave an example of a steady Laplacian soliton that does no satisfy the ERP condition. We show that the steady Laplacian soliton found in Section
\ref{sec-sol} is not an ERP-structure either. These steady Laplacian solitons are not equivalent.

\section{$G_2$-geometry on $G_{A_1,A,B,C}$}\label{sec-ABC}

In this section we explore a large family of Lie groups with parameters $A_1\in\glg_2(\RR)$ and $A,B,C\in\glg_4(\RR)$. We fix a left-invariant
$G_2$-structure on the Lie group, determined by a positive $3$-form $\vp$ on the Lie algebra, and we compute the formulas for the Laplacian and Ricci
operators, as well as the torsion form formulas in terms of the coefficients of $A_1$, $A$, $B$ and $C$.

\subsection{Linear Algebra}

Given a $7$-dimensional Lie algebra $\ggo$, a $3$-form $\vp$ in $\ggo$ is said to be \emph{positive} if there exists a basis
 $\{e_1,\dots,e_7\}$ of $\ggo$ such that
\begin{equation}\label{phi}
\vp=e^{127}+e^{347}+e^{567}+e^{135}-e^{146}-e^{236}-e^{245}.
\end{equation}
$\vp$ determines an inner product $\ip$ and a volume form on $\ggo$, such that $\{e_1,\dots,e_7\}$ turns out to be oriented and orthonormal. In particular,
we can consider the Hodge star operator on $k$-forms:
$$
\ast:\Lambda^k\ggo^*\rightarrow\Lambda^{7-k}\ggo^*, \qquad \beta\wedge\ast\alpha=\la\beta,\alpha\ra e^{1\dots7},
$$
for any $\alpha\in\Lambda^{k}\ggo^*$, $\beta\in\Lambda^{7-k}\ggo^*$.

We denote by $\ggo_1$ the $4$-dimensional subspace generated by $\{e_3,e_4,e_5,e_6\}$ and define a basis $\Upsilon$ for $\Lambda^2\ggo_1^*$ as follows,
\begin{align}
\overline{\omega}_7:=&e^{34}-e^{56},\quad\overline{\omega}_1=e^{35}+e^{46},\quad\overline{\omega}_2:=-e^{36}+e^{45},\label{base-2for}\\
\omega_7:=&e^{34}+e^{56},\quad\omega_1:=e^{35}-e^{46},\quad\omega_2:=-e^{36}-e^{45}.\nonumber
\end{align}
Note that
$$
\vp=e^{127}+\omega_7\wedge e^7+\omega_1\wedge e^1+\omega_2\wedge e^2.
$$
One can easily check that $\Upsilon=\{\overline{\omega}_7,\overline{\omega}_1,\overline{\omega}_2,\omega_7,\omega_1,\omega_2\}$ is an orthogonal basis of
$\Lambda^2\ggo_1^*$ such that every element has norm equal to $2$. Moreover, the Hodge star operator restricted to $\ggo_1$, denoted by
$\ast_{\ggo_1}:\Lambda^k\ggo_1^*\rightarrow \Lambda^{4-k}\ggo_1^*$, acts on each element of $\Upsilon$ in the following way,
$$
\ast_{\ggo_1}\overline{\omega}_i=-\overline{\omega}_i, \qquad \ast_{\ggo_1}\omega_i=\omega_i, \qquad i=1,2,7.
$$

Let $\theta$ denote the derivative of the action of $\Gl(\ggo_1)$ on $\Lambda^k\ggo_1^*$, we mean the representation
$$\theta:\glg(\ggo_1)\longrightarrow\End(\Lambda^k\ggo_1^*),$$
such that
\begin{equation}\label{def-tita}
\theta(M)\alpha(\cdot,\dots,\cdot)=-\alpha(M\cdot,\dots,\cdot)-\dots-\alpha(\cdot,\dots,M\cdot),
\end{equation}
for every $M\in\glg(\ggo_1)$ and $\alpha\in\Lambda^k\ggo_1^*$. In particular, $\theta(M)$ is a derivation of $\Lambda^k\ggo_1^*$ for any $M$. It follows
easily that,
\begin{equation}\label{sls-astg1}
\ast_{\ggo_1}\theta(M)\alpha=-\theta(M^t)\ast_{\ggo_1}\alpha-\tr{M}\ast_{\ggo_1}\alpha, \quad \forall M\in\glg_4(\RR).
\end{equation}

Hence, for every matrix $M\in\glg(\ggo_1)\equiv\glg_4(\RR)$, one obtains that $\theta(M)$ can be written as follows in terms of the basis $\Upsilon$,
\begin{equation}\label{titaM-conTraza}
\theta(M)=\left[\begin{array}{c|c} M_1-\tfrac{\tr{M}}{2}\id & M_2  \\  \hline &\\
                                   M_2^t & M_4-\tfrac{\tr{M}}{2}\id \\
\end{array} \right],\qquad M_1^t=-M_1,\; M_4^t=-M_4,
\end{equation}
for some $M_1,M_2,M_4\in\glg_3(\RR)$. Note that when $\tr{M}=0$, $\theta$ defines the classical isomorphism between $\slg_4(\RR)$ and $\sog(3,3)$.
\begin{remark}
$\tr{\theta(M)}=-3\tr{M}$, for every $M\in\glg_4(\RR)$.
\end{remark}

For the $7$-dimensional Lie algebra $\ggo$ and the positive $3$-form $\vp$ as in \eqref{phi}, we consider the \emph{Hodge Laplacian operator} on $k$-forms
defined by,
$$
\Delta_k:\Lambda^k\ggo^*\rightarrow\Lambda^k\ggo^*,\qquad \Delta\alpha=(-1)^k\left(d\ast d\ast-\ast d\ast d\right)\alpha, \qquad
\forall\alpha\in\Lambda^k\ggo^*.
$$

On the other hand, according to the following irreducible $G_2$-module decompositions (see \cite[(2.14)]{Bry} for a description of the summands),
\begin{align*}
\Lambda^2\ggo^*=&\Lambda^2_7\ggo^*+\Lambda^2_{14}\ggo^*,\\
\Lambda^3\ggo^*=&\Lambda^3_1\ggo^*+\Lambda^2_{7}\ggo^*+\Lambda^2_{27}\ggo^*,
\end{align*}
Bryant proved that we can decompose $d\vp\in\Lambda^4\ggo^*$ and $d\ast\vp\in\Lambda^5\ggo^*$ in the following way:
$$
d\vp=\tau_0\ast\vp+3\tau_1\wedge\vp+\ast\tau_3,\qquad d\ast\vp=4\tau_1\wedge\ast\vp+\tau_2\wedge\vp,
$$
where $\tau_0\in\RR$, $\tau_1\in\Lambda^1\ggo^*$, $\tau_2\in\Lambda^2_{14}\ggo^*$ and $\tau_3\in\Lambda^3_{27}\ggo^*$ are the \emph{torsion forms} of
$\vp$. In \cite[(3)]{MOV}, the authors gave the following useful formulas for the torsion forms:
\begin{align}
\tau_0=\tfrac{1}{7}\ast(d\vp\wedge\vp), & \quad \tau_2=-\ast d\ast\vp+4 \ast (\tau_1\wedge\ast\vp),\label{tor-for}\\
 \tau_1=-\tfrac{1}{12}\ast(\ast
d\vp\wedge\vp), & \quad \tau_3=\ast d\vp-\tau_0\vp-3\ast(\tau_1\wedge\vp). \nonumber
\end{align}

\subsection{The family $G_{A_1,A,B,C}$}

Let $\ggo$ be a $7$-dimensional Lie algebra with basis $\{e_1,\dots,e_7\}$ and Lie bracket determined by,
\begin{equation}\label{muABC}
A_1=\ad{e_7}|_{\spann\{e_1,e_2\}}=\left[\begin{array}{cc}x&z\\
y&w\end{array}\right],\quad A=\ad{e_7}|_{\ggo_1},\quad B=\ad{e_1}|_{\ggo_1},\quad C=\ad{e_2}|_{\ggo_1},
\end{equation}
such that $\spann\{e_1,e_2\}$ is abelian, $\ggo_0:=\spann\{e_7,e_1,e_2\}$ is a subalgebra, $\ggo_1:=\spann\{e_3,e_4,e_5,e_6\}$ is an abelian ideal and
$\hg:=\spann\{e_1,\dots,e_6\}$ is unimodular, that is, $\tr B=\tr C=0$. We further require that,
$$
[A,B]=x\,B+y\,C,\qquad [A,C]=z\,B+w\,C,\qquad [B,C]=0,
$$
in order to satisfy the Jacobi condition. We denote by $G_{A_1,A,B,C}$, the simply connected Lie group with Lie algebra $\ggo$. It follows that
$G_{A_1,A,B,C}$ is solvable and the nilradical of $\ggo$ has dimension greater than or equal to $4$.

A \emph{$G_2$-structure} on a differentiable manifold is a differentiable $3$-form such that it is positive at every point of $M$ (see \eqref{phi}). On a
Lie group, a left-invariant $G_2$-structure is determined by its value at the identity. In particular, we consider on each $G_{A_1,A,B,C}$ the
left-invariant $G_2$-structure defined by the positive $3$-form $\vp\in\Lambda^3\ggo^*$ as in \eqref{phi}.

One can attempt to study certain properties or flows (such as the Laplacian flow or the Laplacian co-flow among others) on this large family of
left-invariant $G_2$-structures $\{(G_{A_1,A,B,C},\vp)\}$. For this reason it is convenient to have some formulas, such as $\ast d\ast d\vp$ and $d\ast
d\ast \vp$, needed to calculate the Hodge Laplacian and the torsion forms of $\vp$ in terms of the matrices $A_1$, $A$, $B$ y $C$.

Recall from the definition of $\theta(M):\Lambda^1\ggo_1^*\rightarrow\Lambda^1\ggo_1^*$ for $M\in\glg_4(\RR)$ (see \eqref{def-tita}) that
\begin{equation}\label{sls-tita}
\theta(M)e^{i+2}=-\sum_{j=1}^4 M_{ij} e^{j+2}, \qquad i\in\{1,\dots,4\}.
\end{equation}
The following proposition summarizes the formulas that one needs in order to compute the exterior derivative of any $k$-form in the Lie algebra $\ggo$,
depending on $A_1$, $A$, $B$ and $C$.
\begin{proposition}\label{sls-prop}
Let $\alpha\in\Lambda^i\ggo_1^*$ and $\beta\in\Lambda^j\ggo_0^*$, it follows that
\begin{itemize}
\item[{\rm (i)}] $d\alpha=(-1)^i \left(\theta(A)\alpha\wedge e^7+\theta(B)\alpha\wedge e^1+\theta(C)\alpha\wedge e^2\right)$.
\item[{\rm (ii)}] $d\,e^1=-\theta(A_1)e^1\wedge e^7=(x\,e^1+z\,e^2)\wedge e^7$.
\item[{\rm (iii)}] $d\,e^2=-\theta(A_1)e^2\wedge e^7=(y\,e^1+w\,e^2)\wedge e^7$.
\item[{\rm (iv)}] $d\, e^7=0$.
\item[{\rm (v)}] $\ast(\alpha\wedge\beta)=(-1)^{i\,j}\ast_{\ggo_1}\alpha\wedge\ast_{\ggo_0}\beta$.
\end{itemize}
\end{proposition}

\begin{proof}
It is sufficient to prove (i) for $1$-forms since the exterior derivative is a derivation. From the definition of the Lie bracket, for $i=1,2,3,4$ we have
that
$$
[e_7,e_{i+2}]=A\,e_i=\sum_{j=1}^4 A_{ji} e^{j+2} \Rightarrow \la d\,e^j,e^{7(i+2)}\ra=-A_{ji}.
$$
Analogously for $B$ and $C$. Hence,
$$
d\,e^{j+2}=\sum_{i=1}^4 \left(A_{ji} e^{(i+2)7}+B_{ji} e^{(i+2)1}+C_{ji} e^{(i+2)2}\right) = -\theta(A)e^j\wedge e^7-\theta(B)e^j\wedge
e^1-\theta(C)e^j\wedge e^2,
$$
for every $j=1,2,3,4$, which proves that (i) is true for $1$-forms. Items (ii), (iii) can be proved in much the same way. The item (iv) follows directly
from the fact that $\hg$ is an ideal. To prove (v), note that $|\alpha\wedge\beta|^2=|\alpha|^2|\beta|^2$, then
$$
(\alpha\wedge\beta)\wedge(\ast_{\ggo_1}\alpha\wedge\ast_{\ggo_0}\beta)=(-1)^{j(4-i)}\alpha\wedge\ast_{\ggo_1}\beta\wedge\alpha\wedge\ast_{\ggo_0}\beta
                =(-1)^{i\,j}|\alpha|^2e^{3456}\wedge|\beta|^2e^{127},
$$
which completes the proof.
\end{proof}

\subsection{Formulas for $(G_{A_1,A,B,C},\vp)$}

We aim in this section to express in terms of $A_1$, $A$, $B$ and $C$ some formulas needed to calculate the torsion forms and the Laplacian of $\vp$.
Indeed, the following theorem displays such formulas for $d\vp$, $\ast d\vp$, $ d\ast d\vp$ and $\ast d\ast d\vp$. Recall from \eqref{base-2for} the
definition of $\omega_i$ for $i=1,2,7$.

\begin{theorem}\label{sls-teo1}
Let $\ggo$ be a $7$-dimensional Lie algebra with Lie bracket determined by $A_1$, $A$, $B$ and $C$ as in \eqref{muABC}. Consider the $G_2$-structure $\vp$
defined in $\eqref{phi}$, the following formulas hold for the exterior derivative and the Hodge star operator on $\ggo$,
\begin{itemize}
\item[{\rm (i)}] $\vp=e^{127}+\omega_7\wedge e^7+\omega_1\wedge e^1+\omega_2\wedge e^2$,\\
\item[{\rm (ii)}] $d\vp=(\theta(B)\omega_2-\theta(C)\omega_1)\wedge e^{12}+(\theta(B)\omega_7-\theta(A)\omega_1+x\omega_1+y\omega_2)\wedge e^{17}$
\item[]      $\qquad +(\theta(C)\omega_7-\theta(A)\omega_2+z\omega_1+w\omega_2)\wedge e^{27}$, \\
\item[{\rm (iii)}] $\ast d\vp=(-\theta(B^t)\omega_2+\theta(C^t)\omega_1)\wedge e^7 +(\theta(B^t)\omega_7-\theta(A^t)\omega_1-(\tr{A})\omega_1-x\omega_1-y\omega_2)\wedge e^2 $
\item[]      $\qquad +(-\theta(C^t)\omega_7+\theta(A^t)\omega_2+(\tr{A})\omega_2+z\omega_1+w\omega_2)\wedge e^1 $,\\
\item[{\rm (iv)}]  $d\ast d\vp=\theta(B)(-\theta(B^t)\omega_2+\theta(C^t)\omega_1)\wedge e^{17}+\theta(C)(-\theta(B^t)\omega_2+\theta(C^t)\omega_1)\wedge e^{27}$
\item[]      $\quad\qquad+\theta(A)(-\theta(B^t)\omega_7+\theta(A^t)\omega_1+(\tr{A})\omega_1+x\omega_1+y\omega_2)\wedge e^{27}$
\item[]      $\quad\qquad+\theta(B)(\theta(B^t)\omega_7-\theta(A^t)\omega_1-(\tr{A})\omega_1-x\omega_1-y\omega_2)\wedge e^{12}$
\item[]      $\quad\qquad+\theta(A)(\theta(C^t)\omega_7-\theta(A^t)\omega_2-(\tr{A})\omega_2-z\omega_1-w\omega_2)\wedge e^{17}$
\item[]      $\quad\qquad+\theta(C)(\theta(C^t)\omega_7-\theta(A^t)\omega_2-(\tr{A})\omega_2-z\omega_1-w\omega_2)\wedge e^{12}$,\\
\item[{\rm (v)}]   $\ast d\ast d\vp=\theta(B^t)(\theta(B)\omega_2-\theta(C)\omega_1)\wedge e^{2}+\theta(C^t)(\theta(C)\omega_1-\theta(B)\omega_2)\wedge e^{1}$
\item[]      $\qquad\qquad+\theta(A^t)(-\theta(B)\omega_7+\theta(A)\omega_1-x\omega_1-y\omega_2)\wedge e^{1} $
\item[]      $\qquad\qquad+(\tr{A})(-\theta(B)\omega_7+\theta(A)\omega_1-x\omega_1-y\omega_2)\wedge e^{1} $
\item[]      $\qquad\qquad+\theta(A^t)(-\theta(C)\omega_7+\theta(A)\omega_2-z\omega_1-w\omega_2)\wedge e^{2}$
\item[]      $\qquad\qquad+(\tr{A})(-\theta(C)\omega_7+\theta(A)\omega_2-z\omega_1-w\omega_2)\wedge e^{2}$
\item[]      $\qquad\qquad+\theta(B^t)(\theta(B)\omega_7-\theta(A)\omega_1+x\omega_1+y\omega_2)\wedge e^{7}$
\item[]      $\qquad\qquad+\theta(C^t)(\theta(C)\omega_7-\theta(A)\omega_2+z\omega_1+w\omega_2)\wedge e^{7}$.
\end{itemize}
\end{theorem}

\begin{proof}
The item (i) follows directly from the definition of $\vp$ and $\omega_i$'s. By Proposition \ref{sls-prop} we obtain that,
\begin{align*}
d\vp=&d\,e^{12}\wedge e^7+d\omega_7\wedge e^7+d\omega_1\wedge e^1+\omega_1\wedge d\,e^1+d\omega_2\wedge e^2+\omega_2\wedge d\,e^2\\
    =&(\theta(B)\omega_7\wedge e^1+\theta(C)\omega_7\wedge e^2)\wedge e^7+(\theta(A)\omega_1\wedge e^7+\theta(C)\omega_1\wedge e^2)\wedge e^1 \\
     &+\omega_1\wedge (x\,e^1+z\,e^2)\wedge e^7+(\theta(A)\omega_2\wedge e^7+\theta(B)\omega_2\wedge e^1)\wedge e^2+\omega_2\wedge (y\,e^1+w\,e^2)\wedge e^7\\
    =&\theta(B)\omega_7\wedge e^{17}+\theta(C)\omega_7\wedge e^{27}-\theta(A)\omega_1\wedge e^{17}-\theta(C)\omega_1\wedge e^{12}+x\,\omega_1\wedge e^{17}\\
     &+z\,\omega_1\wedge e^{27}-\theta(A)\omega_2\wedge e^{27}+\theta(B)\omega_2\wedge e^{12}+y\,\omega_2\wedge e^{17}+w\,\omega_2\wedge e^{27},
\end{align*}
which proves (ii). In order to prove (iii), we apply Proposition \ref{sls-prop} (v) to the above formula,
\begin{align*}
\ast d\vp =&\ast_{\ggo_1}(\theta(B)\omega_2-\theta(C)\omega_1)\wedge \ast_{\ggo_0}e^{12}+\ast_{\ggo_1}(\theta(B)\omega_7-\theta(A)\omega_1+x\omega_1+y\omega_2)\wedge \ast_{\ggo_0}e^{17}\\
           &+\ast_{\ggo_1}(\theta(C)\omega_7-\theta(A)\omega_2+z\omega_1+w\omega_2)\wedge \ast_{\ggo_0}e^{27}\\
           =&(-\theta(B^t)\omega_2+\theta(C^t)\omega_1)\wedge e^{7}-(-\theta(B^t)\omega_7+\theta(A^t)\omega_1+\tr{A}\omega_1+x\omega_1+y\omega_2)\wedge e^{2}\\
           &+(-\theta(C^t)\omega_7+\theta(A^t)\omega_2+\tr{A}\omega_2+z\omega_1+w\omega_2)\wedge e^{1}.
\end{align*}
The last equality follows from \eqref{sls-astg1}. In the same manner, we can se that (iv) and (v) hold.
\end{proof}

The following result may be proved in much the same way as Theorem \ref{sls-teo1}.
\begin{theorem}\label{sls-teo2}
Let $\ggo$ be a $7$-dimensional Lie algebra with Lie bracket determined by $A_1$, $A$, $B$ and $C$ as in \eqref{muABC} and let $\vp$ be the $G_2$-structure
defined in $\eqref{phi}$, one obtains that
\begin{itemize}
\item[{\rm (i)}] $\ast\vp=e^{3456}+\omega_7\wedge e^{12}+\omega_1\wedge e^{27}-\omega_2\wedge e^{17}$,\\
\item[{\rm (ii)}] $d\ast\vp=-\tr{A}e^{34567}+\left(\theta(A)\omega_7-\tr{A_1}\omega_7+\theta(B)\omega_1+\theta(C)\omega_2\right)\wedge e^{127}$, \\
\item[{\rm (iii)}] $\ast d\ast\vp=-\tr{A}e^{12}+\ast_{\ggo_1}\left(\theta(A)\omega_7-\tr{A_1}\omega_7+\theta(B)\omega_1+\theta(C)\omega_2\right)$,\\
\item[{\rm (iv)}] $d \ast d\ast\vp=\tr{A_1}\tr{A}e^{127}$
\item[]      $\qquad\qquad-\theta(A)((\tr{A_1}+\tr{A})\omega_7+\theta(A^t)\omega_7+\theta(B^t)\omega_1+\theta(C^t)\omega_2)\wedge e^7$
\item[]      $\qquad\qquad-\theta(B)((\tr{A_1}+\tr{A})\omega_7+\theta(A^t)\omega_7+\theta(B^t)\omega_1+\theta(C^t)\omega_2)\wedge e^1$
\item[]      $\qquad\qquad-\theta(C)((\tr{A_1}+\tr{A})\omega_7+\theta(A^t)\omega_7+\theta(B^t)\omega_1+\theta(C^t)\omega_2)\wedge e^2$.
\end{itemize}
\end{theorem}

Let us now state two corollaries of Theorems \ref{sls-teo1} and \ref{sls-teo2}, respectively, which are useful to establish necessary and sufficient
conditions to determine if $\vp$ is closed or coclosed. Such formulas are then applied to calculate, in each case, the torsion forms in terms of $A_1$,
$A$, $B$ and $C$.
\begin{corollary}\label{sls-closed}
$\vp$ is closed if and only if
$$
\theta(A)\omega_1=\theta(B)\omega_7+x\omega_1+y\omega_2,\quad \theta(A)\omega_2=\theta(C)\omega_7+z\omega_1+w\omega_2,\quad
\theta(B)\omega_2=\theta(C)\omega_1.
$$
\end{corollary}
In that case, the only torsion form that survives in \eqref{tor-for} is the $2$-form $\tau_2=-\ast d\ast\vp$ and the Laplacian equals
$\Delta\vp=d\tau_2=-d\ast d\ast\vp$. Both formulas can be obtained from Theorem \ref{sls-teo2} (iii),(iv).

\begin{corollary}
$\vp$ is coclosed if and only if
$$
\tr{A}=0,\qquad \theta(A)\omega_7+\theta(B)\omega_1+\theta(C)\omega_2=(\tr{A_1})\omega_7.
$$
\end{corollary}
When this happens, the surviving torsion forms are $\tau_0=\tfrac{1}{7}\ast(\ast d\vp\wedge\vp)$ and $\tau_3=\ast d\vp-\tau_0\vp$ and the Hodge Laplacian
remains $\Delta\vp=\ast d\tau_3+\tau_0\ast d\vp=\ast d\ast d\vp$, whose formula can be seen in Theorem \ref{sls-teo1} (v).

\subsection{Torsion formulas for $(G_{A_1,A,B,C},\vp)$}

In the general case, beyond the closed and coclosed setting, the torsion forms can be also calculated in terms of $A_1$, $A$, $B$ and $C$. In the following
proposition we summarize the obtained results. We denote by $a_{ij}=\la\ad{e_7}(e_j),e_i\ra$ the coefficients of $A$, and analogously for $B$ and $C$.
\begin{proposition}\label{sls-torprop}
Let $\ggo$ be a $7$-dimensional Lie algebra with Lie bracket defined by $A_1$, $A$, $B$ and $C$ as in \eqref{muABC}. Consider the $G_2$-structure $\vp$
 defined in $\eqref{phi}$, the torsion forms can be calculated as follows,
\begin{itemize}
\item[{\rm (i)}] $\tau_0=\tfrac{2}{7}\left(a_{46}-a_{64}+a_{53}-a_{35}+b_{35}+b_{64}-b_{53}-b_{46}\right.\\
                     \left.\qquad+c_{54}+c_{63}-c_{45}-c_{36}+z-y\right)$,\\
\item[{\rm (ii)}] $\tau_1=-\tfrac{1}{12}(a_{64}+a_{35}-a_{46}-a_{53}+b_{43}+b_{65}-b_{34}-b_{56})\wedge e^2$
\item[]            $\qquad-\tfrac{1}{12}(a_{36}+a_{45}-a_{63}-a_{54}+c_{56}+c_{34}-c_{65}-c_{43})\wedge e^1$
\item[]            $\qquad-\tfrac{1}{12}(b_{63}+b_{54}-b_{36}-b_{45}+c_{46}+c_{53}-c_{64}-c_{35}+2(\tr{A_1}+\tr{A}))\wedge e^7$,\\
\item[{\rm (iii)}]$\tau_2=\tfrac{1}{3}(\tr{A}-2\tr{A_1}+b_{45}+b_{36}-b_{54}-b_{63}+c_{35}+c_{64}-c_{53}-c_{46}) e^{12}$
\item[]     $\quad  +\tfrac{1}{3}(a_{64}+a_{35}-a_{46}-a_{53}+b_{65}+b_{43}-b_{56}-b_{34}) e^{17}$
\item[]     $\quad  +\tfrac{1}{3}(a_{54}+a_{63}-a_{45}-a_{36}+c_{65}+c_{43}-c_{56}-c_{34}) e^{27}$
\item[]     $\quad  +\tfrac{1}{3}(\tr{A_1}-2a_{33}-2a_{44}+a_{55}+a_{66}+2c_{46}-2 c_{35}-2 b_{45}-2 b_{36}-c_{53}+c_{64}-b_{63}-b_{54})e^{34}$
\item[]     $\quad  +\tfrac{1}{3}(-2 a_{54}+2 a_{36}+2 c_{56}+2 c_{34}+a_{63}-a_{45}+c_{65}+c_{43}-3 b_{55}-3 b_{33})e^{35}$
\item[]     $\quad  +\tfrac{1}{3}(-2 a_{64}-2 a_{35}-2 b_{65}+2 b_{34}-a_{46}-a_{53}-b_{56}+b_{43}+3 c_{66}+3 c_{33})e^{36}$
\item[]     $\quad  +\tfrac{1}{3}(a_{64}+a_{35}+b_{65}-b_{34}+2 a_{46}+2 a_{53}+2 b_{56}-2 b_{43}+3 c_{55}+3 c_{44}) e^{45}$
\item[]     $\quad  +\tfrac{1}{3}(-a_{54}+a_{36}+c_{56}+c_{34}+2 a_{63}-2 a_{45}+2 c_{65}+2 c_{43}+3 b_{66}+3 b_{44})e^{46}$
\item[]     $\quad  +\tfrac{1}{3}(\tr{A_1}+a_{33}+a_{44}-2 a_{55}-2 a_{66}-c_{46}+c_{35}+b_{45}+b_{36}+2 c_{53}-2 c_{64}+2 b_{63}+2 b_{54})e^{56}$,\\
\item[{\rm (iv)}]$\tau_3=\tau_0\,e^{127}+(-\theta(B^t)\omega_2+\theta(C^t)\omega_1-\tau_0\omega_7-3\lambda_1\omega_2+3\lambda_2\omega_1)\wedge e^7$
\item[]     $\qquad+(\theta(B^t)\omega_7-\theta(A^t)\omega_1-(\tr{A}+x+3\lambda_7)\omega_1+(-y-\tau_0)\omega_2+3\lambda_1\omega_7)\wedge e^2$
\item[]     $\qquad+(-\theta(C^t)\omega_7+\theta(A^t)\omega_2+(z-\tau_0)\omega_1+(\tr{A}+w+3\lambda_7)\omega_2-3\lambda_2\omega_7)\wedge e^1$,
\end{itemize}
where $\lambda_1:=\la\tau_1,e^1\ra$, $\lambda_2:=\la\tau_1,e^2\ra$ and $\lambda_7:=\la\tau_1,e^7\ra$.
\end{proposition}

To prove the proposition, we first state the following result, which follows from \eqref{sls-tita}.
\begin{remark}\label{titaM}
For each $M\in\glg_4(\RR)$,  with coefficients $M=[m_{ij}]$ for $i,j\in\{3,4,5,6\}$, we get,
\begin{align*}
\theta(M)\omega_7=&-(m_{33}+m_{44})e^{34}+(m_{63}-m_{45})e^{35}-(m_{46}+m_{53})e^{36}\\
                  &+(m_{64}+m_{35})e^{45}+(m_{36}-m_{54})e^{46}-(m_{55}+m_{66})e^{56},\\
\theta(M)\omega_1=&-(m_{54}+m_{63})e^{34}-(m_{33}-m_{55})e^{35}+(m_{43}-m_{56})e^{36}\\
                  &+(m_{65}-m_{34})e^{45}+(m_{44}+m_{66})e^{46}+(m_{45}+m_{36})e^{56},\\
\theta(M)\omega_2=&\,(m_{64}-m_{53})e^{34}+(m_{43}+m_{65})e^{35}+(m_{33}+m_{66})e^{36}\\
                  &+(m_{44}+m_{55})e^{45}+(m_{56}+m_{54})e^{46}+(m_{35}-m_{46})e^{56}.
\end{align*}
\end{remark}

\begin{proof}
By Theorem \ref{sls-teo1} (i) and (ii), we can prove that,
\begin{align*}
d\vp\wedge\vp=&\left((\theta(B)\omega_2-\theta(C)\omega_1)\wedge\omega_7-(\theta(B)\omega_7-\theta(A)\omega_1+x\omega_1+y\omega_2)\wedge \omega_2\right.\\
              &\left.+(\theta(C)\omega_7-\theta(A)\omega_2+z\omega_1+w\omega_2)\wedge\omega_1\right)\wedge e^{127}\\
             =&2\left(\theta(A)\omega_1\wedge\omega_2+\theta(B)\omega_2\wedge\omega_7+\theta(C)\omega_7\wedge\omega_2+z-y\right)\wedge e^{127},
\end{align*}
then it follows from \eqref{tor-for} that,
$$
\tau_0=\tfrac{1}{7}\ast(d\vp\wedge\vp)=\tfrac{2}{7}\ast_{\ggo_1}(\theta(A)\omega_1\wedge\omega_2+\theta(B)\omega_2\wedge\omega_7+\theta(C)\omega_7\wedge\omega_2+z-y),
$$
and the requested formula for $\tau_0$ follows by applying Remark \ref{titaM}.

For $\tau_1$, we first note that if $\alpha\in\Lambda^2\ggo_1^*$ is such that $\ast_{\ggo_1}\alpha=\alpha$, then
$\ast_{\ggo_1}\beta\wedge\alpha=\beta\wedge\alpha$, for any $\beta\in\Lambda^2\ggo_1^*$. Indeed,
$$
\ast_{\ggo_1}\beta\wedge\alpha=\la\beta,\alpha\ra e^{3456}=\beta\wedge\ast_{\ggo_1}\alpha=\beta\wedge\alpha.
$$
By theorem \ref{sls-teo1}, we have that,
\begin{align*}
\ast
d\vp\wedge\vp=&\left(-\ast_{\ggo_1}(\theta(B)\omega_2-\theta(C)\omega_1)\wedge\omega_1+\ast_{\ggo_1}(\theta(C)\omega_7-\theta(A)\omega_2+z\omega_1+w\omega_2)\wedge\omega_7 \right)\wedge e^{17}\\
                   &-\left(\ast_{\ggo_1}(\theta(B)\omega_2-\theta(C)\omega_1)\wedge\omega_2+\ast_{\ggo_1}(\theta(B)\omega_7-\theta(A)\omega_1+x\omega_1+y\omega_2)\wedge\omega_7\right)\wedge e^{27}\\
                   &+\ast_{\ggo_1}(\theta(B)\omega_7-\theta(A)\omega_1+x\omega_1+y\omega_2)\wedge\omega_1\wedge e^{12}\\
                   &+\ast_{\ggo_1}(\theta(C)\omega_7-\theta(A)\omega_2+z\omega_1+w\omega_2)\wedge\omega_2\wedge e^{12}\\
                  =&(\theta(B)\omega_1+\theta(A)\omega_7)\wedge\omega_2\wedge e^{17}-(\theta(C)\omega_2+\theta(A)\omega_7)\wedge\omega_1 \wedge e^{27}\\
                   &-(\theta(B)\omega_1+\theta(C)\omega_2)\wedge\omega_7\wedge e^{12}+2(\tr{A}+\tr{A_1})e^{123456}.
\end{align*}
Hence, from \eqref{tor-for} it follows that,
\begin{align*}
\tau_1=&-\tfrac{1}{12}\ast(\ast d\vp\wedge\vp)\\
      =&-\la\theta(B)\omega_1+\theta(A)\omega_7,\omega_2\ra e^{2}-\la\theta(C)\omega_2+\theta(A)\omega_7,\omega_1\ra e^{1}\\
       &-\la\theta(B)\omega_1+\theta(C)\omega_2,\omega_7\ra e^{7}+2(\tr{A}+\tr{A_1})e^{7},
\end{align*}
and using the result given in Remark \ref{titaM}, we obtain the desired conclusion for $\tau_1$.

For simplicity of notation, we named $\lambda_1$, $\lambda_2$ and $\lambda_7$ the coefficients of $\tau_1$ such that $\tau_1=\lambda_1 e^1+\lambda_2
e^2+\lambda_7 e^7$; we therefore obtain that,
$$
\ast(\tau_1\wedge\ast_\vp)=\lambda_1 e^{27}-\lambda_2 e^{17}+\lambda_7 e^{12}+\lambda_1\omega_1+\lambda_2\omega_2+\lambda_7\omega_7,
$$
which in addition to Theorem \ref{sls-teo2} (iii) imply that
\begin{align*}
\tau_2=&-\ast d\ast\vp+4 \ast (\tau_1\wedge\ast\vp),\\
      =&\tr{A} e^{12}+(\tr{A_1}+\tr{A})\omega_7+\theta(A^t)\omega_7+\theta(B^t)\omega_1+\theta(C^t)\omega_2\\
      &+4\left(\lambda_1 e^{27}-\lambda_2 e^{17}+\lambda_7 e^{12}+\lambda_1\omega_1+\lambda_2\omega_2+\lambda_7\omega_7\right)\\
      =&(\tr{A}+4\lambda_7) e^{12}+4\lambda_1 e^{27}-4\lambda_2 e^{17}+(\tr{A_1}+\tr{A}+4\lambda_7)\omega_7\\
      &+4\lambda_1\omega_1+4\lambda_2\omega_2+\theta(A^t)\omega_7+\theta(B^t)\omega_1+\theta(C^t)\omega_2.
\end{align*}
The expected formula in (iii) follows by applying Remark \ref{titaM} to the above equality.

To conclude, we calculate
$$
\ast(\tau_1\wedge\vp)=(\lambda_1\omega_2-\lambda_2\omega_1)\wedge e^{7}-(\lambda_1\omega_7-\lambda_7\omega_1)\wedge e^{2}+
(\lambda_2\omega_7-\lambda_7\omega_2)\wedge e^{1},
$$
and by \eqref{tor-for}, it follows that,
\begin{align*}
\tau_3=&\ast d\vp-\tau_0\vp-3\ast(\tau_1\wedge\vp)\\
      =&\tau_0\,e^{127}+(-\theta(B^t)\omega_2+\theta(C^t)\omega_1-\tau_0\omega_7-3\lambda_1\omega_2+3\lambda_2\omega_1)\wedge e^7\\
       +&(\theta(B^t)\omega_7-\theta(A^t)\omega_1-(\tr{A}+x+3\lambda_7)\omega_1+(-y-\tau_0)\omega_2+3\lambda_1\omega_7)\wedge e^2\\
       +&(-\theta(C^t)\omega_7+\theta(A^t)\omega_2+(z-\tau_0)\omega_1+(\tr{A}+w+3\lambda_7)\omega_2-3\lambda_2\omega_7)\wedge e^1,
\end{align*}
which completes the proof.\end{proof}

Note that if we apply the results given in Remark \ref{titaM} we will obtain a precise formula for $\tau_3$ in terms of the coefficients of $A_1$, $A$, $B$
and $C$ which define the Lie algebra $\ggo$.

\subsection{Ricci formula for $(G_{A_1,A,B,C},\vp)$}\label{ricci-sec}

Another operator that can be computed in terms of $A_1$, $A$, $B$ and $C$ is the Ricci operator, which is useful to prove, for example, that two
$G_2$-structures are not equivalent. In terms of the orthogonal decomposition $\ggo_0\oplus\ggo_1=\spann\{e_7,e_1,e_2\}\oplus\spann\{e_3,e_4,e_5,e_6\}$,
one can calculate the Ricci operator according to \cite[(25)]{solvsolitons} and obtain $\Ricci|_{\ggo_0\times\ggo_1}=0$,
$$\Ricci|_{\ggo_1}=\frac{1}{2} \left([A,A^t]+ [B,B^t]+ [C,C^t]\right)- (\tr{A_1}+\tr{A})S_A,$$
and
$$
\Ricci|_{\ggo_0}=\left[\begin{array}{c|c}
&   \\
-\tr(S_A^2) & -\tr(S_A\,B)\qquad \qquad -\tr (S_A\,C)   \\  &\\ \hline &  \\
 \begin{matrix}
                 -\tr (S_A\,B)  \\
                 -\tr (S_A\,C)\end{matrix}
                &        - \left[\begin{matrix}
                 \tr (S_B^2)  & \tr (S_B\,C)\\
                 \tr (S_B\,C)  &  \tr (S_C^2)
                \end{matrix}\right]+ \frac{1}{2} [A_1,A_1^t]-(\tr{A_1}+\tr{A})S_{A_1}\\ & \\
\end{array}\right],
$$
where $S_M=S(M)$ denotes the symmetric part of the matrix $M$, in other words $S_M=\tfrac{M+M^t}{2}$.

\section{Eigenforms}\label{sec-EF}

In this section, we apply the formulas obtained in the above section to prove that there is no any closed eigenform on the family of $G_2$-structures
$\{(G_{A_1,A,B,C},\vp)\}$, such that the torsion $2$-form is of the form $\tau_2=a\,e^{12}+b\,e^{34}+c\,e^{56}$.

On a differentiable manifold $M$, a $G_2$-structure $\vp$ is said to be an \emph{eigenform} if
\begin{equation}\label{EF}
 \Delta\vp=\lambda \vp, \qquad \mbox{for some } \lambda\in\RR.
\end{equation}

If in addition $\vp$ is closed, Lauret proved in \cite[Lemma 3.4]{LS-ERP} that in the homogeneous case $\lambda=|\tau|^2/7$. The intriguing question is
whether there exists a closed $G_2$-structure that is also an eigenform. For this reason, we study condition \eqref{EF} on the large family of
left-invariant $G_2$-structures $\{(G_{A_1,A,B,C},\vp)\}$ defined in the above section.

Recall from Corollary \ref{sls-closed} that if $\vp$ is closed, then the only surviving torsion form is $\tau_2$.
\begin{proposition} A closed eigenform $(G_{A_1,A,B,C},\vp)$ such that the torsion form is $\tau_2=a\,e^{12}+b\,e^{34}+c\,e^{56}$, for some $a+b+c=0$, must be torsion-free.
\end{proposition}

\begin{proof}
Suppose that $\ggo$ is the Lie algebra with Lie bracket defined by $A_1$, $A$, $B$ and $C$ as in \eqref{muABC}. Consider the matrices $\theta(A)$,
$\theta(B)$, $\theta(C)\in\glg_6(\RR)$ in the basis $\Upsilon$ of $\Lambda^2\ggo^*$. From \eqref{titaM-conTraza}, we can write
$$
\theta(A)=\left[\begin{smallmatrix}
0&a_{12}&a_{13}   & a_{14}&a_{15}&a_{16}\\
-a_{12}&0&a_{23}  & a_{24}&a_{25}&a_{26}\\
-a_{13}&-a_{23}&0 & a_{34}&a_{35}&a_{36}\\
a_{14}&a_{24}&a_{34}& 0 & a_{45} & a_{46}\\
a_{15}&a_{25}&a_{35}& -a_{45} & 0 & a_{56} \\
a_{16}&a_{26}&a_{36}& -a_{46} & -a_{56} &0
\end{smallmatrix}\right]-\tfrac{\tr{A}}{2}\id.
$$
The same notation can be used to write $\theta(B)$ and $\theta(C)$ with coefficients $b_{ij}$ and $c_{ij}$, respectively. Note that, for $B$ and $C$, the
multiple of the identity vanishes in both cases since $\hg$ is unimodular and so $\tr{B}=\tr{C}=0$. From Corollary \ref{sls-closed}, the condition for
$\vp$ to be closed implies that,
$$
\theta(B)^6=\theta(C)^5,\quad \theta(A)^5=\theta(B)^4+x\,\id^5+y\,\id^6,\quad  \theta(A)^6=\theta(C)^4+z\,\id^5+w\,\id^6,
$$
where the superscripts denote the column vector of the matrix. It is immediate that
\begin{align*}
b_{j4}=a_{j5},\quad c_{j4}=a_{j6},\quad c_{j5}=b_{j6},\quad b_{56}=c_{56}=a_{45}=a_{46}= 0,  \\
b_{46}=c_{45}=\tfrac{z+y}{2},\quad b_{45}=\tfrac{\tr{A}}{2}+x,\quad c_{46} =\tfrac{\tr{A}}{2}+w,\quad a_{56} =\tfrac{z-y}{2},
\end{align*}
for $j=1,2,3$. Recall from Corollary \ref{sls-closed} and Theorem \ref{sls-teo2}, that $\tau=\tr{A}e^{12}+\alpha$ for
$$
\alpha=(x+w+\tr{A})\omega_7+\theta(A^t)\omega_7+\theta(B^t)\omega_1+\theta(C^t)\omega_2.
$$
Hence, if we assume $\tau=a\,e^{12}+b\,e^{34}+c\,e^{56}$, then $\tr{A}=a$ and
$$
\alpha=\tfrac{b-c}{2}\,\overline{\omega}_7+\tfrac{b+c}{2}\,\omega_7.
$$
From this and the fact that $a+b+c=0$, it follows that $\tr{A}=-x-w$ and
$$
c_{16}=-a_{14}-b_{15}+b+\tfrac{\tr{A}}{2},\quad c_{26}=-a_{24}-b_{25},\quad c_{36} = -a_{34}-b_{35},\quad c = -\tr{A}-b.
$$
If in addition $\vp$ is an eigenform, then exists $\lambda\in\RR$ such that $\Delta\vp=d\tau=\lambda\,\vp$. Moreover, by \cite[Lemma 3.4]{LS-ERP} one
obtains that $\lambda=\tfrac{|\tau|^2}{7}$. Therefore, by applying the formula of Theorem \ref{sls-teo2} (iv), we obtain $\lambda=-(x+w)\tr{A}=\tr{A}^2$
and
\begin{align*}
\lambda\,\omega_7=&\theta(A)\alpha=\theta(A)\left(\tfrac{b-c}{2}\,\overline{\omega}_7+\tfrac{b+c}{2}\,\omega_7\right),\\
\lambda\,\omega_1=&\theta(B)\alpha=\theta(B)\left(\tfrac{b-c}{2}\,\overline{\omega}_7+\tfrac{b+c}{2}\,\omega_7\right),\\
\lambda\,\omega_2=&\theta(C)\alpha=\theta(C)\left(\tfrac{b-c}{2}\,\overline{\omega}_7+\tfrac{b+c}{2}\,\omega_7\right).
\end{align*}
In particular,
$$
0=\la\theta(B)\alpha,\overline{\omega}_7\ra=\la\alpha,\theta(B)^t\overline{\omega}_7\ra=\la
\left(b+\tfrac{\tr{A}}{2}\right)\id^1-\tfrac{\tr{A}}{2}\id^4,\theta(B)^1\ra=-\tfrac{\tr{A}}{2}\,a_{15},
$$
where again the superscripts denote the columns of the matrix of the operator in the basis $\Upsilon$. Since $\lambda=0$ implies that $\tau=0$, we can
assume that $\tr{A}$ does not vanish, thus $a_{15}=0$. The same argument follows from
$$
0=\la\theta(C)\alpha,\overline{\omega}_7\ra=-\tfrac{\tr{A}}{2}\,a_{16},
$$
in order to prove that $a_{16}=0$. We can also calculate
$$
0=\la\theta(A)\alpha,\overline{\omega}_7\ra=-\tfrac{1}{2}\tr{A}\left(b+\tfrac{\tr{A}}{2}+a_{14}\right),
$$
and since $\tr{A}\neq 0$, thus $a_{14}=-b-\tfrac{\tr{A}}{2}$. Using the following equation,
$$
\tr{A}^2=\lambda=\la\theta(A)\alpha,\omega_7\ra=-b^2-\tr{A}\,b,
$$
we obtain that $\tr{A}=b=$, which is a contradiction and completes the proof.
\end{proof}

\section{Solitons}\label{sec-sol}

In this section we give a new family of Laplacian solitons that provides the second known family of shrinking Laplacian solitons as well as a second
example of a steady Laplacian soliton that is not ERP. In order to do so, we first recall some insights about solitons, not only for the Laplacian flow,
but also for the Ricci flow.  We also overview two examples of Laplacian solitons: one consists of an one-parameter family of Laplacian solitons given by
Lauret, and the second one is a steady Laplacian soliton given by Fino and Raffero.

\subsection{Preliminaries}
\begin{definition}\cite{solvsolitons}
Given a Lie algebra $\ggo$ and and inner product $\ip$ on $\ggo$, we say that $(\ggo,\ip)$ is an (algebraic) \emph{Ricci soliton} if there exist $c\in\RR$
and  $D\in\Der(\ggo)$ such that:
\begin{equation}\label{ricsol}
\Ricci=c \id + D,
\end{equation}
where $\Ricci$ is the Ricci operator of the left-invariant metric induced by $\ip$ on the simply connected Lie group $G$ with Lie algebra $\ggo$. We say
that the Ricci soliton is expanding if $c<0$, steady if $c=0$ or shrinking if $c>0$.
\end{definition}
We denote by $\ricci$ the Ricci tensor, it is proved in \cite{solvsolitons} that,
$$
\ricci(g)=c g-\tfrac{1}{2}\lca_{X_D}g,
$$
for $X_D$ the left-invariant vector field on the simply connected Lie group $G$ with Lie algebra $\ggo$, defined by
\begin{equation}\label{XD}
X_D(p)=\ddt\big{|}_0 f_t(p), \quad \forall p\in G,
\end{equation}
where $f_t\in\Aut(G)$ is the unique automorphism such that $df_t|_e=e^{tD}$. In particular, $(G,\vp)$ is a self-similar solution, that is
$$
\vp(t)=c(t)f(t)^*\vp, \quad\mbox{for some } c(t)\in\RR^* \mbox{ and } f(t)\in\Aut(G),
$$
for the Ricci flow:
$$\dpar
g(t)=-2\Ricci(g(t)).
$$

In 1992, Bryant introduced the \emph{Laplacian flow} for $G_2$-structures given by,
\begin{equation}\label{LF}
\dpar\vp(t) = \Delta_{\vp(t)}\vp(t),
\end{equation}
where $\vp(t)$ is a $1$-parameter family of closed $G_2$-structures on a $7$-dimensional differentiable manifold $M$. It is known that a closed
$G_2$-structure $\vp$ on $M$ flows in a self-similar way, in the sense that the solutions $\vp(t)$ have the form
$$
\vp(t)=c(t)f(t)^*\vp, \quad\mbox{for some } c(t)\in\RR^* \mbox{ and } f(t)\in\Diff(M),
$$
if and only if
$$
\Delta_\vp\vp=c\vp+\lca_{X}\vp, \qquad \mbox{for some}\quad c\in\RR, \quad X\in\mathfrak{X}(M)\; \mbox{(complete)},
$$
where $\lca_X$ denotes the Lie derivative along the field $X$, in which case $c(t)=\left(\frac{2}{3}ct+1\right)^{3/2}$. Analogous to the terminology used
in the Ricci flow theory, we call $\vp$ a {\it Laplacian soliton} and we say that it is {\it expanding}, {\it steady} or {\it shrinking}, if $c>0$, $c=0$
or $c<0$, respectively.

In the particular case where $M=G$ is a simply connected Lie group with Lie algebra $\ggo$ and $\vp$ is a left-invariant form on $G$, one has the following
more friendly definition to work with. Note that a left-invariant $G_2$-structure on $G$ is determined by its value at the identity, or equivalently, by a
$3$-form on the Lie algebra $\ggo$.
\begin{definition}\label{preli-sol}\cite{LF,Hom-sol} Given a $7$-dimensional Lie algebra $\ggo$ and $\vp$ a left-invariant $G_2$-structure on $G$, we say that $(\ggo, \vp)$ is a semi-algebraic {\it Laplacian soliton} if there exist $D\in\Der(\ggo)$
and $c\in\RR$ such that:
\begin{equation}\label{eq-sol}
\Delta\vp = c\,\vp+\lca_{X_D}\vp,
\end{equation}
where $X_D$ is the left-invariant field on $G$ defined as in \eqref{XD}. Equivalently, $\vp$ is a semi-algebraic Laplacian soliton on $G$ if there exist
$c(t)\in\RR^*$ and $f(t)\in\Aut(G)$ such that:
$$
\vp(t)=c(t)f(t)^*\vp
$$
is a solution to the Laplacian flow $\ddt\vp(t)=\Delta\vp$.
\end{definition}
Note that if $(\ggo,\vp)$ is a (semi-algebraic) Laplacian soliton then $(G,\vp)$ is a Laplacian soliton. To simplify notation, we continue to write $\vp$
for the $3$-form on the Lie algebra $\ggo$.

\begin{lemma}\label{preli-lieder}
$\lca_{X_D}\alpha=-\theta(D)\alpha$, for every $\alpha\in\Lambda^k\ggo^*$.
\end{lemma}

\begin{proof}
It is sufficient to prove that the assertion holds only for $\alpha\in\Lambda^1\ggo^*$, since both are derivations of  $\Lambda^k\ggo^*$. Indeed,
$$
\lca_{X_D}e^i(X)=\ddt\big{|}_0 f_t^*\, e^i(X)=\ddt\big{|}_0 e^i(df_t\, X)= \ddt\big{|}_0 e^i(e^{tD}\, X)
                =e^i\left(\ddt\big{|}_0 e^{tD}\, X\right)=e^i(D\, X),
$$
and the lemma follows.
\end{proof}

Another special class of $G_2$-structures was introduced by Bryant in \cite[Remark 13]{Bry} and we recall it next.
\begin{definition}\label{ERP-def}
On a $7$-dimensional differentiable manifold $M$, a closed $G_2$-structure $\vp$ on $M$ is said to be \emph{extremally Ricci pinched} (ERP) if,
\begin{equation}\label{ERP-eq}
\Delta\vp=d\tau=\tfrac{1}{6}|\tau|^2+\tfrac{1}{6}\ast(\tau\wedge\tau),
\end{equation}
for $\tau=-\ast d\ast\vp$, the torsion $2$-form of $\vp$.
\end{definition}
In \cite{ERP2}, it is proved that up to equivalence, there are only five left-invariant ERP-structures on simply connected Lie groups  and that they are
all expanding Ricci solitons and steady Laplacian solitons.

We consider the following invariant functional on the space of all non-flat homogeneous $G_2$-structures, where $\scalar$ denotes the scalar curvature,
\begin{equation}\label{F}
F=\frac{\scalar^2}{|\Ricci|^2}.
\end{equation}
In general, $F$ is less than or equal to $7$ by Cauchy-Schwartz, and equality holds if and only if the metric is Einstein. Thus, $F$ measures how far is
the metric from being Einstein. Bryant proved in \cite[Remark 13]{Bry} that $F$ is less than or equal to $3$ in the compact case and that evaluating at any
ERP $G_2$-structure $F$ equals to $3$. Nevertheless, this estimate does not hold in the general homogeneous case, examples of closed $G_2$-structures on
solvable Lie groups such that $F >  3 $ were found in \cite[Theorem 1.2]{LS-ERP}. In the following section, we give a new family of examples of shrinking
Laplacian solitons which also satisfy that $F$ is bigger than $3$. The functional $F$ is mostly useful to distinguish $G_2$-structures since it is
invariant up to equivalence and scaling.

\subsection{Shrinking solitons}

Nowadays, there are lots of examples of steady and expanding Laplacian solitons (see \cite{FFM,FinRff5,FinRff1,Lin,BF,LF,LS-ERP,N}), but this is not the
case for shrinking Laplacian solitons. Previous to this work, there was in the literature only a one-parameter family of examples of shrinking Laplacian
solitons given by Lauret, which we recall and analyze in the following example.

\begin{example}{\rm\cite[Example 4.10]{LS-ERP}}\label{gj}
Let $\{\sg_a:a\geq 0\}$ be the family of non isomorphic solvable Lie algebras with basis $\{e_1,\dots,e_7\}$ and Lie bracket given by:
$$
[e_1,e_3]=-e_6,\; [e_1,e_4]=-e_5,\; [e_2,e_3]=-e_5,\; [e_2,e_4]=e_6,\; [e_7,e_i]=(A_a)_{ii}e_i,
$$
where $A_a:=\tfrac{1}{4}\Diag\left(1+4\,a,1+4\,a,1-4\,a,1-4\,a,2,2\right)$. Consider the $G_2$-structure $\vp\in\Lambda^3\sg_a^*$ defined as in
\eqref{phi}, which turns out to be closed for every $a\in\RR$. The following assertions were  proved in \cite[Example 4.10]{LS-ERP},
\begin{itemize}
\item $(\sg_a,\vp)$ is a shrinking, steady or expanding Laplacian soliton if $a<\tfrac{3}{4}$, $a=\pm\tfrac{3}{4}$ or $\tfrac{3}{4}<a$, respectively. The multiple of the identity and the derivation for which \eqref{eq-sol} holds are respectively given by $c_a=-\tfrac{9}{2}+8a^2$ and
$D_a=\tfrac{1}{8}\Diag(15-8a-16a^2,15-8a-16a^2,15+8a-16a^2,15+8a-16a^2,30-32a^2,30-32a^2,0). $\\
\item $(\sg_a,\ip)$ is a Ricci soliton if and only if $a=\tfrac{3}{4}$. The multiple of the identity for which \eqref{ricsol} stands is $-3$ and
so $(\sg_{3/4},\ip)$ is an expanding Ricci soliton.\\
\item The functional $F$ defined in \eqref{F} is given by:
$$
F(a)=\frac{\scalar_a^2}{|\Ricci_a|^2}=\frac{(27+16a^2)^2}{153+352a^2+256a^4}.
$$
It satisfies that $F(0)=\frac{81}{17}\approx 4,76$ and $F(a)=3$ if and only if $(\sg_a,\vp)$ is a steady Laplacian soliton. The graphic of $F$ can be seen
in Figure \ref{graficos}.
\end{itemize}
\end{example}

The lack of examples of shrinking Laplacian solitons, which are the ones producing the only known solutions to the Laplacian flow that explode at
finite-time, motivated us to look for new examples. In order to do so, we explore the family of $G_2$-structures $\{(G_{A_1,A,B,C},\vp)\}$, for which we
have the formulas obtained in Section \ref{sec-ABC}. The searching for new examples was successful and the resulting family of examples (non equivalent to
the previous ones) is exhibited in the following lemma.

\begin{lemma}\label{gs-lemma} Consider the family of solvable simply connected Lie groups $G_s$, with corresponding Lie algebras $\ggo_s$ for $s\in\RR$, with
basis $\{e_1,\dots,e_7\}$ and Lie bracket given by
$$
[e_1,e_3]=-e_6,\;\; [e_1,e_4]=-e_5,\;\; [e_2,e_3]=-e_5, \;\; [e_7,e_i]=(A_s)_{ii}e_i,
$$
for $i=1,\dots,6$ and $A_s=\Diag{\left(\tfrac{3}{8}+s,-\tfrac{1}{8}+s,\tfrac{3}{8}-s,-\tfrac{1}{8}-s,\tfrac{1}{4},\tfrac{3}{4}\right)}$. For the
$G_2$-structure $(G_s,\vp)$, where $\vp$ is defined as in \eqref{phi} the following properties hold:
\begin{itemize}
\item[{\rm (i)}] $\vp$ is closed.\\
\item[{\rm (ii)}] $\tau_s=\tfrac{5-8s}{4}\,e^{12}+\tfrac{5+8s}{4}\,e^{34}-\tfrac{5}{2}\,e^{56}$ is the torsion $2$-form of  $\vp$.\\
\item[{\rm (iii)}] $\Delta\vp=\tfrac{64s^2-32s-5}{16}\,e^{127}+\tfrac{64s^2+32s-5}{16}\,e^{347}+\tfrac{5}{2}\,(e^{135}-e^{146}-e^{236}+e^{567})$.\\
\item[{\rm (iv)}] $(G_s,\vp)$ is equivalent to $(G_{-s},\vp)$ for any $s\in\RR$.\\
\item[{\rm (v)}] For $s,t\geq0$, $(G_s,\vp)$ is equivalent to a multiple of $(G_t,\vp)$ if and only if $s=t$.
\end{itemize}
\end{lemma}

\begin{remark}
The simply connected Lie group $G_s$ coincides with the Lie group $G_{A_1,A,B,C}$ defined in Section \ref{sec-ABC} for the following matrices:
$$
A_1=\left[\begin{smallmatrix}\tfrac{3}{8}+s&\\&-\tfrac{1}{8}+s \end{smallmatrix}\right],\, A=\left[\begin{smallmatrix}
\tfrac{3}{8}-s&&&\\&-\tfrac{1}{8}-s&&\\&&\tfrac{1}{4}&\\&&&\tfrac{3}{4}
\end{smallmatrix}\right],\,
B=\left[\begin{smallmatrix}&&0&0\\&&0&0\\0&-1&&\\-1&0&&\end{smallmatrix}\right],\, C=\left[\begin{smallmatrix}&&0&0\\&&0&0\\-1&0&&\\0&0&&
\end{smallmatrix}\right].
$$
\end{remark}

\begin{proof}
It is immediate to check, using Remark \ref{titaM}, that
$$
\begin{array}{c}
\theta(A)\omega_1=\left(s-\tfrac{5}{8}\right)\omega_1, \quad\theta(A)\omega_2=\left(\tfrac{9}{8}-s\right)e^{36}+\left(\tfrac{1}{8}-s\right)e^{45},\\
 \theta(B)\omega_7=-\omega_1,\quad \theta(B)\omega_2=0, \quad \theta(C)\omega_7=e^{36}, \quad\theta(C)\omega_1=0,
\end{array}
$$
hence (i) comes directly from Corollary  \ref{sls-closed}. To prove (ii) we use the formula given in Proposition \ref{sls-torprop} (iii):
\begin{align*}
\tau=&\tfrac{1}{3}(\tr{A}-2\tr{A_1}-b_{54}-b_{63}-c_{53}) e^{12}\\
       &+\tfrac{1}{3}(\tr{A_1}-2a_{33}-2a_{44}+a_{55}+a_{66}-c_{53}-b_{63}-b_{54})e^{34}\\
       &+\tfrac{1}{3}(\tr{A_1}+a_{33}+a_{44}-2 a_{55}-2 a_{66}+2 c_{53}+2 b_{63}+2 b_{54})e^{56},
\end{align*}
replacing with the values given in the previous remark we obtain the desired equality. In order to prove  (iii) we first compute the following:
$$
\theta(A^t)\omega_7=\left(2s-\tfrac{1}{4}\right)e^{34}-e^{56},\quad \theta(B^t)\omega_1=-2\,e^{56},\quad \theta(C^t)\omega_2=-e^{56},
$$
then we name
$$
\alpha:=(\tr{A_1}+\tr{A})\omega_7+\theta(A^t)\omega_7+\theta(B^t)\omega_1+\theta(C^t)\omega_2=\left(2s+\tfrac{5}{4}\right)e^{34}-\tfrac{5}{2}e^{56},
$$
and so
$$
\theta(A)\alpha=\left(4s^2+2s-\tfrac{5}{16}\right)e^{34}-\tfrac{5}{2}e^{56},\quad \theta(B)\alpha=-\tfrac{5}{2}\left(e^{46}-e^{35}\right),\quad
\theta(C)\alpha=-\tfrac{5}{2}e^{36}.
$$
By Corollary \ref{sls-closed}, for calculating the Laplacian it is sufficient to compute the formula given in Theorem \ref{sls-teo2} (iv):
$$
\Delta\vp=-\tr{A_1}\tr{A}e^{127}+\theta(A)\alpha\wedge e^7+\theta(B)\alpha\wedge e^1+\theta(C)\alpha\wedge e^2.
$$
Replacing with the previous calculation one obtains the given assertion. Item (iv) becomes true since
$$
h_s:(\ggo_s,\vp)\rightarrow (\ggo_{-s},\vp),\quad
h_s:=\left[\begin{smallmatrix}&\id_2&&\\\id_2&&&\\&&-\id_2&\\&&&1 \end{smallmatrix}\right],
$$
defines an isomorphism between both Lie algebras $\ggo_s$ and $\ggo_{-s}$ such that $h_s\cdot\vp=\vp$, or equivalently $h_s\in G_2$. This implies that
$(G_s,\vp)\simeq (G_{-s},\vp)$, for any $s\in\RR$, since both Lie groups are unimodular completely solvable. For the proof of (v),  we calculate the
functional defined in \eqref{F}, which is invariant under equivalence and scaling:
$$
F(s):=\frac{\scalar_s^2}{|\Ricci_s|^2}=\frac{(75+64s^2)^2}{1725+4224s^2+4096s^4}.
$$
One can compute the derivative of $F$:
$$
\frac{d}{ds}F(s)=-\frac{1536\,s\,(75+64\,s^2)(125+448\,s^2)}{(1725+4224\,s^2+4096\, s^2)^2}<0, \qquad \forall s> 0,
$$
therefore $F$ is strictly decreasing and thus injective in $\RR\geq0$, and so item (v) follows.
\end{proof}

\begin{remark}
More information about the functional $F$ at $\{(G_s,\vp)\}$ is given in Corollary \ref{coro-fs}.
\end{remark}

\begin{theorem}\label{teo-fs}
Let $\{(G_s,\vp):s\geq 0\}$ be the family of left-invariant closed $G_2$-structures defined in Lemma \ref{gs-lemma}, the following assertions hold,
\begin{itemize}
\item[{\rm (i)}]  $(G_s,\vp)$ is a shrinking, steady or expanding Laplacian soliton depending on whether
$s\in\left[0,\tfrac{\sqrt{15}}{8}\right)$, $s=\tfrac{\sqrt{15}}{8}$ or $s\in\left(\tfrac{\sqrt{15}}{8},\infty\right)$, respectively; with constant
$c_s=-\tfrac{15}{8}+8s^2$ and derivation $D_s=\tfrac{1}{32}\Diag( 45-32s-64s^2,5-32s-64s^2,
45+32s-64s^2, 5+32s-64s^2, 50-128s^2, 90-128s^2,0)$.\\
\item[{\rm (ii)}] $(G_s,\ip)$ is an expanding Ricci soliton if and only if $s=\frac{5}{8}$.\\
\item[{\rm (iii)}] For any $a,s\in\RR$, the Lie algebras $\ggo_s$ and $\sg_a$ are not isomorphic, where $\sg_a$ denotes the Lie algebra given in Example \ref{gj}.
\end{itemize}
\end{theorem}

\begin{remark}
It follows from \cite[(17)]{Hom-sol} (or also from \cite[(23)]{LF}) that for each $s\geq 0$, the corresponding family of self-similar solutions to the
Laplacian flow \eqref{LF} is given by
\begin{align*}
\vp_r(t)=&\left(\left(\tfrac{16}{3}s^2-\tfrac{5}{4}\right)t+1\right)^{\tfrac{3}{2}}\left(e^{r(t)\left(\tfrac{25}{16}-2s-4s^2\right)}e^{127}+e^{r(t)\left(\tfrac{25}{16}+2s-4s^2\right)}e^{347}+e^{r(t)\left(\tfrac{35}{8}-8s^2\right)}e^{567}\right.\\
        &\left.+e^{r(t)\left(\tfrac{35}{8}-8s^2\right)}e^{135}-e^{r(t)\left(\tfrac{35}{8}-8s^2\right)}e^{146}-e^{r(t)\left(\tfrac{35}{8}-8s^2\right)}e^{236}-e^{r(t)\left(\tfrac{15}{8}-8s^2\right)}e^{245}
\right),
\end{align*}
where
$$
r(t)=\left\{\begin{array}{ll} \left(\tfrac{16}{3}s^2-\tfrac{5}{4}\right)^{-1}\log\left(\left(\tfrac{16}{3}s^2-\tfrac{5}{4}\right)t+1\right), &
s\neq\tfrac{\sqrt{15}}{8},\\
&\\
t, & s=\tfrac{\sqrt{15}}{8}. \end{array}\right.
$$
It follows immediately that the solution is defined for $t\in\left(-\infty,\left(\tfrac{5}{4}-\tfrac{16}{3}s^2\right)^{-1}\right)$ if
$s\in\left[0,\tfrac{\sqrt{15}}{8}\right)$, for $t\in(-\infty,\infty)$ if $s=\tfrac{\sqrt{15}}{8}$, and it is defined for
$t\in\left(\left(\tfrac{5}{4}-\tfrac{16}{3}s^2\right)^{-1},\infty\right)$ when $s\in\left(\tfrac{\sqrt{15}}{8},\infty\right)$.
\end{remark}

\begin{proof}
It is straightforward to check that \eqref{eq-sol} is satisfied for $c_s=-\tfrac{15}{8}+8s^2$ and $D_s=\tfrac{1}{32}\Diag( 45-32s-64s^2,5-32s-64s^2,
45+32s-64s^2, 5+32s-64s^2, 50-128s^2, 90-128s^2,0)\in\Der(\ggo_s)$. Therefore, $c<0$ for $s<\tfrac{\sqrt{15}}{8}$, $c=0$ for $s=\tfrac{\sqrt{15}}{8}$ and
$c>0$ if $s>\tfrac{\sqrt{15}}{8}$, which completes the proof of (i). To prove (ii) we first calculate the Ricci operator, using the formula given in
Section \ref{ricci-sec}:
$$
\Ricci_s=\tfrac{1}{16}\Diag\left(-25-24s,-5-24s,-25+24s,-5+24s,10,-10,-15-64s^2\right).
$$
We are interested in finding $\lambda_s\in\RR$ such that $\Ricci_s-\lambda_s \id\in\Der\ggo_s$. It follows immediately from $[e_1,e_3]=-e_6$ that
$\lambda_s=-\tfrac{5}{2}$. Also, since it must vanish at $e_7$, we obtain that $s=\tfrac{5}{8}$. In other words, $\left(\ggo_{5/8},\ip\right)$ is an
expanding Ricci soliton with
$$
\Ricci_{\tfrac{5}{8}}=-\tfrac{5}{2}\id+\tfrac{5}{8}\Diag(0,2,3,5,5,3,0),
$$
which completes the proof of (ii).

Recall that $\hg$ is the $6$-dimensional subspace with basis $\{e_1,\dots,e_6\}$. The proof of (iii) follows from the fact that $A_a$ and $A_s$ are
derivations of $\hg$ and so they must be conjugated (up to scaling) by an automorphism of $\hg$. Hence, they must be conjugated restricted to the center
which is $\spann\{e_5,e_6\}$. Clearly, this can not happen because $A_a\big{|}_{\la e_5,e_6\ra}$ has two equal eigenvalues and $A_s\big{|}_{\la
e_5,e_6\ra}$ has two different eigenvalues.  This completes the proof of the proposition.
\end{proof}

\begin{remark}
An alternative proof of $\sg_a$ and $\ggo_s$ being non isomorphic follows from the classification of $6$-dimensional $2$-step nilpotent Lie algebras
 given for example in \cite[Table 2]{Will}. The nilradical of $\sg_a$ turns out to be isomorphic to the twenty-ninth Lie algebra of the table, while the nilradical of
$\ggo_s$ is isomorphic to the twenty-eighth one.
\end{remark}

\begin{corollary}\label{coro-fs}
Recall from the proof of Lemma \ref{gs-lemma} (v) the functional $F$ at the family of $G_2$-structures $\{(G_s,\vp):s\geq 0\}$:
$$
F(s)=\frac{\scalar_s^2}{|\Ricci_s|^2}=\frac{(75+64s^2)^2}{1725+4224s^2+4096s^4};
$$
the following assertions hold,
\begin{itemize}
\item $F(s)$ reaches its maximum at   $s=0$ and  $F(0)= \frac{75}{23}\approx 3.26 > 3$.\\
\item At the steady Laplacian soliton $(G_{\sqrt{15}/8},\vp)$, we have that $F(s)=\frac{135}{49}\approx 2.75$.
\item At the expanding Ricci soliton $(G_{5/8},\ip)$, we obtain $F(s)=2.5$.
\end{itemize}
\end{corollary}

\begin{figure}\centering
\begin{tikzpicture}
  \draw[->] (-1,0) -- (10,0) ;
  \draw[->] (0,-1) -- (0,5) ;
  \draw[color=blue,domain=0:3] plot (\x,{(0.04545*\x*\x*\x*\x+2.454545*\x*\x+33.13636)/(0.04545*\x*\x*\x*\x+\x*\x+6.954545)});
  \draw[color=green,domain=3:10] plot (\x,{(0.04545*\x*\x*\x*\x+2.454545*\x*\x+33.13636)/(0.04545*\x*\x*\x*\x+\x*\x+6.954545)});
  \draw (3,-0.05) -- (3,0.05) node[below] {{\tiny $\frac{3}{4}$}};
  \draw[color=red,fill] (3,3) circle (1pt);
  \draw (-0.05,4.7647) -- (0.05,4.7647) node[right] {{\tiny $4,76$}};
  \draw[dashed] (0,1) -- (0,1);
  \draw[dashed] (0.3,1) -- (5,1);
  \draw (-0.05,1) -- (0.02,1) node[right] {{\tiny $1$}};
  \draw[color=blue,domain=0:1.93649] plot (\x,{(0.33333*\x*\x*\x*\x+12.5*\x*\x+117.1875)/(0.33333*\x*\x*\x*\x+5.5*\x*\x+35.9375)});
  \draw[color=green,domain=1.93649:10] plot (\x,{(0.33333*\x*\x*\x*\x+12.5*\x*\x+117.1875)/(0.33333*\x*\x*\x*\x+5.5*\x*\x+35.9375)}) ;
  \draw (1.93649,-0.05) -- (1.93649,0.05) node[below] {{\tiny $\frac{\sqrt{15}}{8}$}};
  \draw[color=red, fill] (1.93649,2.75551) circle (1pt);
  \draw (-0.05,3) -- (0.05,3) node[right] {{\tiny $3$}};
  \draw (-0.05,2.75551) -- (0.05,2.75551) node[right] {{\tiny $2,75$}};
  \draw (-0.07,3.260869) -- (0.07,3.260869) node[ right] {{\tiny $3,26$}};
  \draw[dashed] (0.3,1) -- (10,1);
  \draw (-0.05,1) -- (0.02,1) node[right] {{\tiny $1$}};
  \draw (-0.05,2.5) -- (0.05,2.5) node[right] {{\tiny $2,5$}};
  \draw (2.5,-0.05) -- (2.5,0.05) node[below] {{\tiny $\frac{5}{8}$}};
  \draw[fill] (2.5,2.5) circle (0.5pt);
  \draw (10,1) node[right]{$F(s)$};
  \draw (10,1.5) node[right]{$F(a)$};
  \end{tikzpicture}
  \caption{Graph of the functional $F$ at $(\sg_a,\vp)$ and $(\ggo_s,\vp)$.}\label{graficos}
  \end{figure}
Figure \ref{graficos} shows the graphics of $F$ defined in \eqref{F} at the families of $G_2$-structures $\{(\sg_a,\vp):a>0\}$ (see Example \ref{gj}) and
$\{(\ggo_s,\vp):s>0\}$ (see Lemma \ref{gs-lemma}). The green color in both graphics stands for the values of $F$ where the corresponding $G_2$-structure is
an expanding Laplacian soliton. On the contrary, the blue color indicates the points for which $(\sg_a,\vp)$ and $(\ggo_s,\vp)$ are shrinking Laplacian
solitons and the red points denote the steady Laplacian solitons.

\subsection{Steady solitons}
On the other hand, previous to this work there was only one known example, which we recall below, given by Fino and Raffero in \cite{FinRff3}, of a closed
$G_2$-structure which is a steady Laplacian soliton but it is not an ERP-structure, in other words, it does not satisfy \eqref{ERP-eq}.
\begin{example}\label{fr}{\rm\cite[Section 4]{FinRff3}} Let $\ggo_{FR}$ denote the solvable Lie algebra with basis $\{e_1,\dots,e_7\}$ and Lie bracket given
by,
$$
[e_1,e_4]=-2e_5,\;[e_2,e_4]=2e_6,\;[e_7,e_i]=(A_{FR})_{ii}e_i,
$$
where $A_{FR}:=\Diag\left(0,0,1,-1,-1,-1,0\right)$. Fino and Raffero proved that the $3$-form $\vp$ given in \eqref{phi} turns out to be closed and the
$G_2$-structure $(\ggo_{FR},\vp)$ is a steady Laplacian soliton (i.e.\ $ d\tau=\Delta\vp=\lca_{X_D}\vp$), for $D=\Diag(0,0,-4,4,4,4,0)\in\Der(\ggo_{FR})$.
Moreover, they proved that $\vp$ does not satisfy the ERP condition, in fact
$$
-8(e^{146}+e^{245}-e^{567})=d\tau\ne\tfrac{1}{6}|\tau|^2\vp+\tfrac{1}{6}\ast(\tau\wedge\tau)=4\vp+\tfrac{4}{3}(e^{567}-2e^{127}-2e^{347}).
$$
\begin{remark}
This example proves that the converse of  \cite[Corollary 1.2]{ERP1} does not hold, indeed it proves that not every left-invariant steady Laplacian soliton
is ERP.
\end{remark}
\end{example}

Recall from Theorem \ref{teo-fs} the steady Laplacian soliton $(G_{\sqrt{15}/8},\vp)$. A trivial verification shows that \eqref{ERP-eq} does not hold,
thus, $(G_{\sqrt{15}/8},\vp)$ is a steady Laplacian soliton which is not ERP. Indeed,
\begin{align*}
d\tau=&\tfrac{5-2\sqrt{15}}{8} e^{127}+\tfrac{5+2\sqrt{15}}{8} e^{347}+\tfrac{5}{2}(e^{135}-e^{146}-e^{236}+e^{567}),\\
\tfrac{1}{6}(|\tau|^2\vp+\ast(\tau\wedge\tau))=&\tfrac{20-5\sqrt{15}}{24}e^{127}+\tfrac{20+5\sqrt{15}}{24}e^{127}+\tfrac{15}{8}(e^{135}-e^{146}-e^{236}-e^{245})+\tfrac{25}{12}e^{567}.
\end{align*}
In other words, $(G_{\sqrt{15}/8},\vp)$ provides the second example that proves that the converse of \cite[Corollary 1.2]{ERP1} does not hold.

\begin{remark} The $G_2$-structures $(G_{\sqrt{15}/8},\vp)$ from Theorem \ref{teo-fs} and $(G_{FR},\vp)$ from Example \ref{fr} are not equivalent, since the Lie algebras $\ggo_{\sqrt{15}/8}$  and $\ggo_{FR}$ are not isomorphic.
\end{remark}
Indeed, the derivations of $\hg$, $A_{\sqrt{15}/8}$ and $A_{FR}$, should be conjugated (up to scaling) by an automorphism of $\hg$, but $A_{FR}|_\hg$ has
three eigenvalues each one of multiplicity two, meanwhile $A_{\sqrt{15}/8}|_\hg$ has six different eigenvalues. The non equivalence follows from the fact
that both Lie groups are unimodular and completely solvable.


\begin{thebibliography}{MMM}

\bibitem[B]{Bry} {\sc R. Bryant}, Some remarks on $G_2$-structures, Proc. G\"okova Geometry-Topology Conference (2005), 75-109.

\bibitem[FFM]{FFM} {\sc M. Fern\'andez, A. Fino, A. Raffero}, Laplacian flow of closed $G_2$-structures inducing nilsolitons, {\it  J. Geom. Anal.} {\bf 26} (2016), 1808-1837.

\bibitem[FR1]{FinRff5} {\sc  A. Fino, A. Raffero}, Closed warped $G_2$-structures evolving under the Laplacian flow, {\it Ann. Scuola Norm. Sup. Pisa Cl. Sci.} (2017), doi: $10.2422/2036-2145.201709\_004$.

\bibitem[FR2] {FinRff1} {\sc A. Fino, A. Raffero}, Closed $G_2$-structures on non-solvable Lie groups, {\it Rev. Matem. Complutense}, in press.

\bibitem[FR3]{FinRff3} {\sc A. Fino, A. Raffero}, Remarks on homogeneous solitons of the $G_2$-Laplacian flow, to appear in Comtes Rendus Mathematique.

\bibitem[KL]{KL} {\sc I. Kath, J. Lauret}, A new example of a compact ERP $G_2$-structure, preprint 2020 (arXiv).

\bibitem[L1]{solvsolitons} {\sc J. Lauret}, Ricci soliton solvmanifolds, {\it J. reine angew. Math.} {\bf 650} (2011), 1-21.

\bibitem[L2]{BF}  {\sc J. Lauret}, Geometric flows and their solitons on homogeneous spaces, {\it Rendiconti del Seminario Matematico di Torino}, {\bf 74} (2016), 55-93.

\bibitem[L3]{LF}  {\sc J. Lauret}, Laplacian flow of homogeneous $G_2$-structures and its solitons, {\it Proc. London Math. Soc.} {\bf 114} (2017), 527-560.

\bibitem[L4]{LS-ERP}  {\sc J. Lauret}, Laplacian solitons: Questions and homogeneous examples, {\it Diff. Geom. Appl.} {\bf 54} (2017), 345-360.

\bibitem[L5]{Hom-sol} {\sc J. Lauret}, The search for solitons on homogeneous spaces, Abel Symposia, Springer, in press (arXiv).

\bibitem[LN1]{ERP1} {\sc J. Lauret, M. Nicolini}, Extremally Ricci pinched $G_2$-structures on Lie groups, {\it Comm. Anal. Geom.}, in press (arXiv).

\bibitem[LN2]{ERP2} {\sc J. Lauret, M. Nicolini}, The classification of ERP $G_2$-structures on Lie groups, {\it Ann. Mat. Pura App.}, in press (arXiv).

\bibitem[Li]{Lin} {\sc C. Lin}, Laplacian solitons and symmetry in $G_2$-geometry, {\it J. Geom. Phys.} {\bf 64} (2013), 111-119.

\bibitem[Lo]{Lty} {\sc J. Lotay}, Geometric flows of $G_2$ structures, {\it Fields Institute Communications}, Springer, in press (arXiv).

\bibitem[LW]{LW} {\sc J. Lotay, Y. Wei}, Laplacian flow for closed $G_2$-structures: Shi-type esimates, uniqueness and compactness, {\it Geom. Funct. Anal.} {\bf 27}
(2017), 165–233.

\bibitem[MOV]{MOV} {\sc V. Manero, A. Otal, R. Villacampa}, Laplacian coflow for warped $G_2$-structures, {\it Diff. Geom. Appl.}, in press.

\bibitem[N]{N} {\sc M. Nicolini}, Laplacian solitons on nilpotent Lie groups, {\it Bull. Belgian Math. Soc.} {\bf 25} (2018), 183-196.

\bibitem[We]{Wei} {\sc Y. Wey}, Laplacian flow for closed $G_2$-structures, {\it Fields Institute Communications}, Springer, in press.

\bibitem[Wi]{Will} {\sc C. Will}, Rank-one Einstein solvmanifolds of dimension $7$, {\it  Diff. Geom. Appl.} {\bf 19} (2003), 307-318.

\end{thebibliography}
\end{document}